\documentclass{article}

\usepackage{amsmath}
\usepackage{amsthm}
\usepackage{amssymb}

\usepackage{url}

\usepackage[all,cmtip]{xy}
\usepackage[colorlinks=true,urlcolor=blue,linkcolor=blue,plainpages=false,pdfpagelabels
]{hyperref}

\newtheorem{theorem}{Theorem}[section]
\newtheorem{proposition}[theorem]{Proposition}
\newcommand{\bthm}{\begin{theorem}}
\newcommand{\ethm}{\end{theorem}}
\newtheorem{lemma}[theorem]{Lemma}
\newcommand{\blem}{\begin{lemma}}
\newcommand{\elem}{\end{lemma}}
\newcommand {\bua} {\begin{eqnarray*}}
\newcommand {\eua} {\end {eqnarray*}}

\newcommand{\beq}{\begin{equation}}
\newcommand{\eeq}{\end{equation}}
\newcommand{\eps}{\varepsilon}
\newcommand{\ds}{\displaystyle}

\newcommand{\be}{\begin{enumerate}}
\newcommand{\ee}{\end{enumerate}}

\newcommand{\N}{\mathbb{N}}

\newcommand{\gluexy}{X \sqcup_{\theta} Y}

\newcommand{\dnxy}{D_n(x)}
\newcommand{\dnxyk}{D_n(x,k)}
\newcommand{\dnixy}{D_{n_i}(x)}
\newcommand{\dnixyk}{D_{n_i}(x,k)}
\newcommand{\dniuxy}{D_{n_{i+1}}(x)}
\newcommand{\dkxy}{D_k(x)}
\newcommand{\dnimxy}{D_{n_{i-1}}(x)}
\newcommand{\dnirxy}{D_{n_{I(R)}}(x)}
\newcommand{\dnjxy}{D_{n_j}(x)}

\newcommand{\mik}{M_{i,k}}
\newcommand{\mjk}{M_{j,k}}

\begin{document}

\title{A note on an ergodic theorem in weakly uniformly convex geodesic spaces}

\author{Lauren\c{t}iu Leu\c{s}tean${}^{1,2}$,  Adriana Nicolae${}^{3,4}$ \\[0.2cm]
\footnotesize ${}^1$ Faculty of Mathematics and Computer Science, University of Bucharest,\\
\footnotesize Academiei 14,  P. O. Box 010014, Bucharest, Romania\\[0.1cm]
\footnotesize ${}^2$ Simion Stoilow Institute of Mathematics of the Romanian Academy,\\
\footnotesize P. O. Box 1-764, RO-014700 Bucharest, Romania\\[0.1cm]
\footnotesize ${}^3$ Department of Mathematics, Babe\c{s}-Bolyai University, \\
\footnotesize  Kog\u{a}lniceanu 1, 400084 Cluj-Napoca, Romania\\[0.1cm]
\footnotesize ${}^4$ Simion Stoilow Institute of Mathematics of the Romanian Academy,\\
\footnotesize Research group of the project PD-3-0152,\\
\footnotesize P. O. Box 1-764, RO-014700 Bucharest, Romania\\[0.1cm]
\footnotesize E-mails:  Laurentiu.Leustean@imar.ro, anicolae@math.ubbcluj.ro
}

\date{}

\maketitle

\begin{abstract}
\noindent Karlsson and Margulis \cite{KarMar99} proved in the setting of uniformly convex geodesic spaces, 
which additionally satisfy a nonpositive curvature condition, an ergodic theorem that focuses on the asymptotic 
behavior of integrable cocycles of nonexpansive mappings over an ergodic measure-preserving transformation. 
In this note we show that this result holds true when assuming a weaker notion of uniform convexity.\\[1mm]
\end{abstract}

\section{Introduction}

One recent research direction in ergodic theory consists in
generalizing classical ergodic theorems to the setting of geodesic
spaces with a sufficiently rich geometry.  Let $Y$ be a geodesic space
and $D \subseteq Y$. Consider $S$ a semigroup of nonexpansive
(i.e. $1$-Lipschitz) self-mappings defined on $D$ and endow $S$ with
the Borel $\sigma$-algebra induced by the compact-open topology on
$S$.  Assume that $(X,\mu)$ is a probability measure space, $T : X \to
X$ is an ergodic measure-preserving transformation and $w :X \to S$ is
a measurable map. Define the cocycle
\[a_n(x) = w(x)w(Tx)\cdots w(T^{n-1}x).\]
Fix $y \in D$ and suppose that $\ds \int_X d(y,a_1(x)y)d\mu(x)<
\infty$.  One can easily see that the sequence
$\displaystyle\left(\int_X d(y, a_n(x)y) d\mu(x)\right)$ is
subadditive and so, by Fekete's subadditive lemma \cite{Fek23}, the
following limit exists
\[0 \le A = \lim_{n \to \infty}\frac{1}{n}\int_X  d(y, a_n(x)y) d\mu(x) = \inf_{n}\frac{1}{n}\int_X  
d(y, a_n(x)y) d\mu(x) < \infty.\] As an immediate application of
Kingman's subadditive ergodic theorem \cite{Kin68} one gets that for
almost every $x\in X$, $ \ds \lim_{n \to \infty}\frac{1}{n}
d(y,a_n(x)y) = A$.

Karlsson and Margulis \cite{KarMar99} proved that if $Y$ is a complete Busemann convex geodesic 
space satisfying a uniform convexity condition (see Section \ref{section-wuc}, (i) for the 
precise definition), the following  holds: 
if $A>0$, then for almost every $x \in X$, there exists a unique geodesic ray $\gamma$ in 
$Y$ starting at $y$ and depending on $x$ such that
$\ds \lim_{n \to \infty}\frac{1}{n}d(\gamma(An),a_n(x)y) = 0$.	
Thus, instead of the convergence of averages as in classical ergodic results, 
one basically obtains for almost every $x \in X$ the existence of a geodesic ray  that issues 
at $y$ such that, as $n \to \infty$, the values of the cocycles $a_n(x)$ at $y$ 
are `close' to this geodesic ray, a property which is referred to in \cite{KarLed11} as ray approximation. 

This result generalizes the multiplicative ergodic theorem of Oseledec
\cite{Ose68} (see also \cite{KarMar99,KarLed11,Kai89}).  A discussion
on the asymptotic behavior of ergodic products of nonexpansive
mappings and isometries of a (proper) metric space and applications
thereof can be found in \cite{KarLed11}.  In a related line,
ergodic-theoretic results in CAT$(0)$ spaces were proved by Es-Sahib
and Heinich \cite{EsSHei99} and Sturm \cite{Stu03} using barycenter
techniques.  Considering a different notion of barycenter map, Austin
\cite{Aus11} proved an extension of the pointwise ergodic theorem for
mappings with values in complete separable CAT$(0)$ spaces.  Following
a similar proof strategy, but defining another suitable notion of
barycenter (that actually recovers the one from \cite{EsSHei99}),
Navas \cite{Nav13} gave an ergodic theorem for mappings taking values
in Busemann spaces.

In this paper we show that the ergodic theorem given in \cite{KarMar99} holds true for 
a more general class of geodesic spaces, by generalizing its proof to these spaces. More precisely, we prove that we can relax the 
uniform convexity assumption on $Y$ used in \cite{KarMar99} as follows: 
$Y$ is said to be {\it weakly uniformly convex} if for any $a \in Y$, $r>0$ and $\eps \in (0,2]$, 
\[
\delta_Y(a,r,\eps) =  \inf\{1-d(a,m(x,y))/r : d(a,x), d(a,y) \le r, d(x,y) \ge \eps r\} > 0,  
\]
where $m(x,y)$ is a midpoint of a geodesic segment from $x$ to $y$.  
We refer to the mapping $\delta_Y$ as the  {\it the modulus of convexity} of $Y$. 
This notion was used by Reich and Shafrir \cite{ReiSha90} in the setting of hyperbolic spaces. 
In addition, we assume that for every $a\in Y$ and $\eps > 0$ there exists $s > 0$ such that 
\begin{equation} \label{cond-wuc-inf}
\inf_{r \ge s} \delta_Y(a,r,\eps) > 0. 
\end{equation}
The main result of the paper is the following.

\bthm\label{main-erg-thm}
Assume that $Y$ is a complete Busemann convex geodesic space that is weakly uniformly convex and
satisfies (\ref{cond-wuc-inf}). If $A>0$, then for almost every $x \in X$, there exists a 
unique geodesic ray $\gamma$ in $Y$ that issues at $y$ and depends on $x$ such that 
$\ds\lim_{n \to \infty}\frac{1}{n}d(\gamma(An),a_n(x)y) = 0$.
\ethm
In Section \ref{section-wuc} we recall different notions of uniform convexity used in 
nonlinear settings, all of them fitting within the class of weakly 
uniform convex geodesic spaces defined as above and satisfying \eqref{cond-wuc-inf}. 
We also discuss a convexity condition that we actually use in the proof and provide an 
example of a weakly uniformly convex geodesic space satisfying \eqref{cond-wuc-inf} which 
is not uniformly convex in the sense of \cite{KarMar99}.
Section \ref{section-proof-main-thm-detailed} contains the proof of our main result.

\section{Weakly uniformly convex geodesic spaces} \label{section-wuc}

Let $(Y,d)$ be a metric space. $Y$ admits midpoints if for every $x,y \in Y$ there exists a point 
$m(x,y) \in Y$ (called a {\it midpoint} of $x$ and $y$) such that $d(x,m(x,y)) = d(y,m(x,y)) = d(x,y)/2$. 
A {\it geodesic path} (resp. {\it geodesic ray}) in $Y$ is a distance-preserving map 
$\gamma:[0,l]\subseteq\mathbb{R}\to Y$  (resp. $\gamma:[0,\infty)\to Y$).  
If $\gamma$ is a geodesic path such that $\gamma(0)=x$, $\gamma(l)=y$, we say that 
$\gamma$ joins $x,y$ and the image of $\gamma$ is a {\it geodesic segment} from $x$ to $y$, 
denoted by $[x,y]$ if there is no confusion.  
$Y$ is a {\it (uniquely) geodesic space} if every two points in $Y$ are joined by a (unique) geodesic path. 
Any complete metric space that admits midpoints is a geodesic space.
A point $z \in Y$ belongs to a geodesic segment $[x,y]$ if and only if 
there exists $t\in [0,1]$ such that $d(x,z)=td(x,y)$ and $d(y,z)=(1-t)d(x,y)$ and we write 
$z=(1-t)x+ty$. For $t=1/2$ we get a midpoint $m(x,y)$. 
See, for instance, \cite{BriHae99} for details on geodesic spaces.

Following \cite{Leu07}, we can  define weak uniform convexity in an alternative way: 
a geodesic space $Y$ is weakly uniformly convex if there exists a mapping
$\eta:Y\times (0,\infty)\times(0,2]\rightarrow (0,1]$ such that for any
$a\in Y, r>0, \varepsilon\in(0,2]$, every $x,y\in Y$ and all geodesic segments $[x,y]$ we have that, 
\begin{eqnarray}
\left.\begin{array}{l}
d(a,x)\le r\\
d(a,y)\le r\\
d(x,y)\ge\varepsilon r
\end{array}
\right\}
& \quad \Rightarrow & \quad d\left(a,m(x,y)\right)\le (1-\eta(a,r,\eps))r.
\end{eqnarray}
Such a mapping $\eta$ is referred to as {\it a modulus of weak uniform convexity}. 
This definition is equivalent to the previous one, with the modulus of  convexity 
$\delta_Y$ giving the largest modulus of weak uniform convexity.  
A modulus $\eta$ is said to be {\it monotone} if it is nonincreasing in the second argument.\\
A weakly uniformly convex geodesic space $Y$ is strictly convex (i.e., for all $a,x,y \in Y$ with 
$x \ne y$, $d\left(a, m(x,y)\right) < \max\{d(a,x),d(a,y)\}$), hence uniquely geodesic. 
Likewise, if $d(a,x),d(a,y)\leq r$, then $d(a,p)\le r$ for all $p\in [x,y]$.  \\
If one can find a modulus $\eta=\eta(r,\eps)$
that does not depend on $a$, we get the stronger notion of {\it uniform convexity}, 
first considered in a nonlinear setting in \cite{GoeSekSta80, GoeRei84}, and $\eta$ is called 
a {\it modulus of uniform convexity}.  

We include next some particular notions of (weak) uniform convexity:
\be

\item \label{def-unif-conv-KM} Karlsson and Margulis \cite{KarMar99} define uniform 
convexity in a metric space $Y$ that admits midpoints as follows: $Y$ is said to be 
uniformly convex if there exists a strictly decreasing and continuous function $g:[0,1]\to[0,1]$  
with $g(0)=1$, 
so that for any $a,x,y\in Y$ and any midpoint $m(x,y)$,
\beq
\frac{d(a,m(x,y))}M \leq g\left(\frac{d(x,y)}{2M}\right), \quad \text{where } M=\max\{d(x,a),d(y,a)\}. 
\label{uc-KarMar}
\eeq
Any geodesic space satisfying the above condition is uniformly convex in our sense. Indeed, given $r>0$,  
$\varepsilon\in(0,2]$, let $a,x,y\in Y$ satisfy $d(x,a)\leq r, d(y,a)\leq r$ and 
$d(x,y)\ge\varepsilon r$. Then 
  $d(a,m(x,y))\leq g(d(x,y)/(2M))M\leq g(\eps/2)r$, so $\eta(\eps) =  1 - g\left(\eps/2\right) > 0$ 
is a modulus of uniform convexity that only depends on $\eps$.

\item Gelander, Karlsson and Margulis \cite{GelKarMar08} consider a strictly 
convex geodesic space $Y$ to be uniformly convex if it is weakly uniformly convex and 
for all $\eps > 0$ there exists $\eta(\eps) > 0$ such that $\delta_Y(a,r,\eps) \ge \eta(\eps)$
 for all $r > 0, a \in Y$.
 
\item In \cite{Kel14}, a metric space $Y$ that admits midpoints is called uniformly 
$p$-convex  (where $p \in [1,\infty]$) if for every $\varepsilon > 0$ there exists 
$\rho_p(\varepsilon) \in (0,1)$ such that for all $a,x,y \in Y$ with 
$d(x,y) > \varepsilon \mathcal{M}^p(d(a,x),d(a,y))$ for $p > 1$ and 
$d(x,y) > |d(a,x) - d(a,y)| + \varepsilon \mathcal{M}^1(d(a,x),d(a,y))$ for $p = 1$ we have that
\beq\label{uc-Kel14}
d(a,m(x,y)) \le (1-\rho_p(\varepsilon))\mathcal{M}^p(d(a,x),d(a,y)),
\eeq
where $\mathcal{M}^p(\alpha,\beta) = (\alpha^p/2 + \beta^p/2)^{1/p}$
and $\mathcal{M}^\infty(\alpha,\beta) = \max\{\alpha,\beta\}$.\\
By \cite[Lemma 1.4]{Kel14}, 
any uniformly $p$-convex space is uniformly $\infty$-convex which, in turn, is uniformly convex 
in our sense. Let $r>0$, $\varepsilon\in(0,2]$ and $a,x,y\in Y$ 
with $d(x,a)\leq r, d(y,a)\leq r$ and $d(x,y)\ge\varepsilon r$. Then $d(x,y) > (\eps M)/2$, 
where $M$ is as in \eqref{uc-KarMar}. This yields $d(a,m(x,y))\leq (1-\rho_\infty(\eps/2))M 
\leq (1-\rho_\infty(\eps/2))r$.
Hence,  $Y$ has a modulus of uniform convexity $\eta(\eps) = \rho_\infty\left(\eps/2\right) > 0$ 
that only depends on $\eps$. 
\ee
Geodesic spaces which are $p$-uniformly convex in the sense of Naor and Silberman \cite{NaoSil11} 
(see also \cite{Kuw14,Oht07}) satisfy all three definitions (i)-(iii). These spaces are defined as follows: 
for a fixed $1<p<\infty$, a geodesic space $Y$ is called {\it $p$-uniformly convex} with parameter 
$k>0$ if for every $a,x,y\in Y$ and $t\in[0,1]$,
\[d(a,(1-t)x + ty)^p\leq(1-t) d(a,x)^p+t d(a,y)^p- \frac{k}2 t(1-t)d(x,y)^p.\]
We see below that $p$-uniformly convex spaces indeed satisfy (i)-(iii). Remark first that, by 
\cite[Proposition 2.5, (1)]{Kuw14}, $k \le c_p$ with 
$$
c_p=\left\{
\begin{array}{ll}
2(p-1) & \mbox{if } p \in (1,2) \\
8/2^p & \mbox{if } p \ge 2.
\end{array}\right.
$$
Let $a \in Y$, $r >0$, $\eps \in (0,2]$. Take $x,y \in Y$ such that $d(a,x) \le r$, $d(a,y) \le r$ and 
  $d(x,y) \ge \eps r$. Then $1 - d(a,m(x,y))/r \ge 1 - (1 - k\eps^p/8)^{1/p}$. This implies (ii).
To see that (i) is satisfied, consider $a,x,y \in Y$ and $M$ as in \eqref{uc-KarMar}.
Then $d(a,m(x,y)) /M\le (1 - (k/8)(d(x,y)/M)^p)^{1/p}$. For $t \in [0,1]$, take 
$\ds g(t) = (1 - k(2t)^p/8)^{1/p}$.  Then (i) holds.
Let $\eps > 0$ and $\rho_p(\eps) = 1 -(1-k\eps^p/8)^{1/p}$. 
One gets that \eqref{uc-Kel14} holds, hence (iii) 
is satisfied (see also \cite[Example, p. 361]{Kel14}).

For the rest of this paper we assume that any weakly uniformly convex geodesic space also 
satisfies condition (\ref{cond-wuc-inf}). From the above arguments it is easy to see that 
all uniformly convex spaces described in (i)-(iii) satisfy (\ref{cond-wuc-inf}). Although 
we state our main result in the setting of weakly uniformly convex 
geodesic spaces that additionally satisfy a nonpositive curvature assumption, we remark 
that what we actually use in the proof is the following convexity condition which holds 
in any weakly uniformly convex geodesic space.

\subsection{A convexity assumption}

We say that a  geodesic space $Y$ has  {\it property $(C)$}
if there exists  a mapping $\Psi:Y\times (0,\infty)\times(0,2]\rightarrow (0,1]$ satisfying:
\be
\item[(C1)] for all $y\in Y, r>0, \varepsilon\in(0,2]$, every $x,z\in Y$ with $d(x,y)=r$, $d(y,z) \ge r$, 
if $w$ belongs to a geodesic segment $[y,z]$ and $d(y,w)=r$,  then
\[r + d(x,z) \le d(y,z) +  \Psi(y,r,\eps)r \quad  \text{ implies} \quad d(w,x) \le \eps r.\]
\item[(C2)] for all $y \in Y, \eps \in (0,2]$, there exists $s > 0$ such that $\inf_{r\ge s} \Psi(y,r,\eps) > 0$.
\ee

\begin{lemma} \label{lemma-wuc-prop}
Any weakly uniformly convex geodesic space $Y$ has property $(C)$ with 
$\Psi(y,r,\eps)=\delta_Y(y,r,\eps)$ for all $y \in Y$, $r > 0$ and $\eps \in (0,2]$.
\end{lemma}
\begin{proof} 
Define $\Psi:=\delta_Y$ and let $y,r,\eps,x,z,w$ as in (C1). We may assume that $w \ne x$. By hypothesis, 
$d(x,z) \le d(y,z) - r + \delta_Y(y,r,\eps)r = d(z,w) + \delta_Y(y,r,\eps)r$. 
Denote $m(w,x)$ by $m$. Then $d(m,z) < \max\{d(z,x), d(z,w)\}\leq d(z,w) + \delta_Y(y,r,\eps)r$, hence
$d(y,z) - d(y,m) < d(z,w) + \delta_Y(y,r,\eps)r$, from where $d(y,m) >
d(y,w)-\delta_Y(y,r,\eps)r=(1-\delta_Y(y,r,\eps))r$. By weak uniform convexity, 
$d(y,m) \le \left(1-\delta_Y\left(y,r,d(x,w)/r\right)\right)r$, hence 
$\delta_Y\left(y, r,d(x,w)/r\right) < \delta_Y(y, r,\eps)$.
As $\delta_Y$ is nondecreasing in the third argument, we get $d(x,w) \le \eps r$.\\
(C2) follows immediately from \eqref{cond-wuc-inf}.
\end{proof}

\subsection{An example} \label{subsection-example}

In the sequel we give an example of a weakly uniformly convex space satisfying \eqref{cond-wuc-inf}
which is not uniformly convex in the sense of Karlsson and Margulis. We refer to \cite{BriHae99} 
for all the undefined notions.

In the construction of our example, we glue spaces through points. Assume that $(X,d_X)$ and $(Y,d_Y)$ are 
geodesic spaces, $\theta \in X$ and $\tau \in Y$. Let $Z$ be the quotient of the disjoint union $X \sqcup Y$ by the 
equivalence relation generated by $[\theta \sim \tau]$. We identify $X$ and $Y$ with their images 
in $Z$, write $Z = \gluexy$ and refer to it as the {\it gluing} of $X$ and $Y$ obtained by identifying the points $\theta, \tau$.  
This is a geodesic metric space with the gluing metric 
\[d(x,y)=
\begin{cases}
d_X(x,y), & \mbox{if } x,y\in X\\
d_Y(x,y), & \mbox{if } x,y\in Y\\
d_X(x,\theta) + d_Y(\tau,y), & \mbox{if } x\in X, y \in Y.
\end{cases}
\]

\begin{lemma} \label{lemma-gluing}
Assume that $(X,d_X)$ and $(Y,d_Y)$ are weakly uniformly convex geodesic spaces with monotone 
moduli  $\eta_X$ and $\eta_Y$, respectively. For every $\theta\in X$, $\tau\in Y$, $\gluexy$ 
is weakly uniformly convex with a modulus 
\bua
\eta(a,r,\eps)=\begin{cases}
\min\{\eta_X(a,r,\eps), \eta_Y(\tau,r,\eps)\eps/2, \eps/4, \eta_X(a,r,\varepsilon/4)
\}, & \text{ if } a\in X\\
\min\{\eta_Y(a,r,\eps), \eta_X(\theta,r,\eps)\eps/2, \eps/4, \eta_Y(a,r,\varepsilon/4)
\}, & \text{ if } a\in Y.
\end{cases}
\eua
\end{lemma}
\begin{proof}
Let $a\in \gluexy$, $\varepsilon \in (0,2]$, $r > 0$ and $x,y\in \gluexy$ be such that 
$d(a,x), d(a,y) \le r$ and $d(x,y)\geq \eps r$.  Let us denote $m(x,y)$ with $m$, for simplicity.
We assume that $a\in X$, the case $a\in Y$ being similar. 
We distinguish the following cases:
\be
\item $x,y \in X$. Then, since $\eta_X$ is a modulus of weak  uniform convexity for $X$,  $d(a,m)\leq (1-\eta_X(a,r,\eps))r$.
\item $x,y \in Y$.  Let $M:=\max\{d(\theta,x), d(\theta,y)\}$. One can easily see that
$\varepsilon r /2\leq M\leq r$ and that $d(a,\theta)+M\leq r$. By the weak uniform convexity of $Y$, 
$d(\theta,m)=d_Y(\tau,m)\leq (1-\eta_Y(\tau,r,\eps))M$.
It follows then that $d(a,m) \leq (1-\eta_Y(\tau,r,\eps)\eps/2)r$.
\item $x \in X$, $y \in Y$.  If $m\in Y$, then $m\in [\theta,y]$, so $d(\theta,y)=d(\theta,m)+d(m,y)$. 
It follows that $d(a,m) = d(a,\theta)+d(\theta,m)=d(a,y) - d(m,y) \le  (1-\varepsilon/2)r$.\\
If  $m\in X$, so $m\in[x,\theta]$, we have two cases. If $d(m, \theta) < \varepsilon r/8$, then $d(\theta,y) > 3\varepsilon r/8$, so
$d(a,\theta)=d(a,y)-d(\theta,y)\leq (1 - 3\varepsilon/8)r$. Thus,  $d(a,m)\leq (1 - \varepsilon/4)r$.
Suppose now that  $d(m,\theta) \ge \varepsilon r/8$ and let  $p\in [x,m]$ be such that $m$ is the midpoint of 
the segment $[p,\theta]$. Then $d(a,p)\le r$ and, since $d(a,\theta)\le r$ and $d(p,\theta)\geq \eps r/4$, apply the fact that 
$X$ is weakly uniformly convex to get that $d(a,m) \leq (1-\eta_X(a,r,\varepsilon/4))r$.
\ee
\end{proof}
\noindent As a consequence, if $X$ and $Y$ have moduli $\eta_X, \eta_Y$ which depend only on $\eps$, then $\gluexy$
has a modulus $\eta$ with the same property.

Consider the $2$-dimensional sphere $\mathbb{S}^2$ which is a geodesic space when endowed with the distance 
$d_{\mathbb{S}^2}(x,y)=\arccos\left(x\mid y\right)$, where $(\cdot\mid \cdot)$ is the Euclidean scalar product. 
Let 
\[b = \left(\sqrt{3}/2, 1/2, 0\right), c = \left(0, 1, 0\right), a_n = \left(1/(2n), \sqrt{3}/(2n), \sqrt{1 - 1/n^2}\right).\] 
Then $m(b,c)= \left(1/2, \sqrt{3}/2, 0\right)$, $d_{\mathbb{S}^2}(b,c) =\pi/3$, 
$d_{\mathbb{S}^2}(a_n,b) = d_{\mathbb{S}^2}(a_n,c) 
= \arccos \left(\sqrt{3}/(2n)\right)$ and $d_{\mathbb{S}^2}\left(a_n,m(b,c)\right) = \arccos \left(1/n\right)$ for all
$n \in \mathbb{N}, n \ge 2$.

Denote by $\Delta_n$ the geodesic triangle $\Delta(a_n,b,c)$. Since we will glue these triangles, 
we denote the points $b$ and $c$ by $b_n$ and $c_n$, respectively. Note that, for each $n \ge 2$, $\Delta_n$ is 
a $2$-uniformly convex geodesic space. In fact, Ohta proved in \cite{Oht07} that, for $\kappa>0$, 
any CAT$(\kappa)$ space $X$  with 
$\text{diam}(X) <\pi/(2\sqrt{\kappa})$ is $2$-uniformly convex with parameter 
$k=(\pi-2\sqrt{\kappa}\,\sigma)\,\tan(\sqrt{\kappa}\,\sigma)$ for any 
$0<\sigma\leq\pi/(2\sqrt{\kappa})-\text{diam}(X)$. In particular, since $\Delta_n$ is a 
CAT(1) space and $\text{diam}(\Delta_n) < \pi/2$ it follows, by the argument given in Section \ref{section-wuc}, 
that it  has  a modulus of uniform convexity  given by 
$\delta_{\Delta_n}(\varepsilon) = \sqrt{1 - (1 - k\varepsilon^2/8)}$.

Glue $\Delta_n$ and $\Delta_{n+1}$ successively through a point by identifying 
$b_n$ with $a_{n+1}$ in order to obtain a chain of triangles $Y = \sum_{n \ge 2}\Delta_n$.  
More precisely, one considers first the gluing $Y_1:=\Delta_2 \sqcup_{b_2} \Delta_3$, then 
one forms $Y_2:=Y_1\sqcup_{b_3} \Delta_4$  and so on, 
for general $n\geq 2$, $Y_{n+1}:=Y_n\sqcup_{b_{n+2}}\Delta_{n+3}$. 
One glues in this way the sequence of triangles $(\Delta_n)$ obtaining the unbounded 
geodesic space $Y$. Applying repeatedly  Lemma \ref{lemma-gluing},  we get that for every $n\ge 1$, 
$Y_n$ is uniformly convex, with a modulus of uniform convexity $\eta_{Y_n}$ which depends only on $\eps$. 

\begin{proposition}
$Y$ is a weakly uniformly convex geodesic space satisfying \eqref{cond-wuc-inf}.
\end{proposition}
\begin{proof}
Let $a\in Y$, $\varepsilon \in (0,2],r > 0$ and $x,y\in Y$ be such that $d(a,x),d(a,y) \le r$ 
and $d(x,y)\geq \eps r$. If $i$ is such that $a\in Y_i$, then $a,x,y\in Y_N$, 
with $N = i +\lceil r/\arccos\left(\sqrt{3}/(2i+4)\right) \rceil$. Thus, $d(a,m(x,y))\leq (1-\eta_{Y_N}(\eps))r$. 

We prove now that \eqref{cond-wuc-inf} holds. Take $s = (i+1)\pi/\varepsilon$ and $r \ge s$. 
Note that since $d(x,y) \ge \varepsilon r$, it follows that $d(x,y) \ge (i+1)\pi$. 
At least one of the points $x$ and $y$ does not belong to 
$Y_i$ and we may assume that this point is $y$. One can also easily see that 
$m(x,y)$ does not belong to $Y_i$ either and that $m(x,y) \in [a,y]$. Thus, as in 
the proof of Lemma \ref{lemma-gluing}, case (iii), we have that $d(a,m(x,y)) \le (1-\varepsilon/2)r$. 
Therefore, $\inf_{r \ge s} \delta_Y(a,r,\varepsilon) \ge \varepsilon/2$.
\end{proof}

\begin{proposition}
$Y$ does not admit a modulus of uniform convexity that does not depend on the center of balls.
\end{proposition}
\begin{proof}
Let $r = \pi/2, \varepsilon = 2/3$ and $\delta \in (0,1]$ be arbitrary. 
Take $n$ sufficiently large such that $\arccos\left(1/n\right) > (1-\delta)\pi/2$. Then $d(b_n,c_n)=\varepsilon r$, 
$d(a_n,b_n) = d(a_n,c_n) \le r$, while $d(a_n, m(b_n,c_n))> (1-\delta)r$.
\end{proof}

\section{Proof of the main theorem}\label{section-proof-main-thm-detailed}

We denote $d(y,a_n(x)y)$ by $\dnxy$ and $d(y,a_{n-k}(T^kx)y)$ by $\dnxyk$. \\
Let $E$ be the set of points $x \in X$ for which taking any $\varepsilon > 0$ there exist 
$M \in \mathbb{N}$ and infinitely many $n \in \mathbb{N}$ such that for every  $k \in [M,n]$,
\[\dnxy - \dnxyk \ge (A - \varepsilon)k.\]
By \cite[Proposition 4.2]{KarMar99} we know that $\mu(E) = 1$.

\begin{lemma} \label{lemma-Ki-ni}
Let $x \in E$ such that $\displaystyle \lim_{n \to \infty}\dnxy/n = A > 0$. Then 
for all sequences $(\alpha_i) \subseteq (0,1], \,  (p_i) \subseteq \mathbb{N}$, there exist sequences 
$(K_i) \subseteq \mathbb{N}, \,(n_i) \subseteq \mathbb{N}$ satisfying for all $i\ge 1$:
\be
\item\label{lemma-Ki-ni-1} $p_i \le K_i$, $n_i\in (K_{i+1},n_{i+1})$ and  $\dniuxy \ge \max\{ \dnixy, An_i\}$;
\item\label{lemma-Ki-ni-2} for all $k \in [K_i, n_i]$, $\left|\dkxy-Ak\right| \le Ak/2^i$ and  
\begin{equation} \label{lemma-Ki-ni-iii-2}
\left(1- \min\{1/2^i,\alpha_i\}\right)\dkxy + d(a_k(x)y,a_{n_i}(x)y) \le  \dnixy.
\end{equation}
\ee
\end{lemma}
\begin{proof}
For $i \ge 1$, take $\eps_i = \min\{A/(1+2^{i+1}), A\alpha_i/(2+\alpha_i)\}$. Then
\begin{equation} \label{lemma-Ki-ni-eq0}
\frac{2\varepsilon_i}{A - \varepsilon_i} \le \min\{1/2^i,\alpha_i\}.
\end{equation}
Since $x \in E$, there exist $M_i$ and infinitely many $n$ satisfying
\begin{equation}\label{lemma-Ki-ni-eq1}
\dnxy - \dnxyk \ge (A - \varepsilon_i)k \qquad \text{ for every }k \in [M_i,n].
\end{equation}
Moreover, since $\ds \lim_{n \to \infty}\dnxy/n = A$, one gets $J_i$ such that for every $k \ge J_i$,
\begin{equation}\label{lemma-Ki-ni-eq2}
(A-\varepsilon_i) k \le \dkxy \le (A + \varepsilon_i)k.
\end{equation}
Take $K_i: = \max\{M_i, J_i, p_i\}$. Then there  are infinitely many $n$ such that 
\eqref{lemma-Ki-ni-eq1} and \eqref{lemma-Ki-ni-eq2} hold for every $k \in [K_i,n]$. 
We define the sequence $(n_i)$ as follows. Take $n_1 > K_1+K_2$ such that \eqref{lemma-Ki-ni-eq1} and 
\eqref{lemma-Ki-ni-eq2} hold for $k \in [K_1,n_1]$. For $i \ge 2$, pick 
$n_i > \max\{n_{i-1}, K_{i+1}\}$ such that $\dnixy \ge \max\{\dnimxy, An_{i-1}\}$ and 
\eqref{lemma-Ki-ni-eq1}, \eqref{lemma-Ki-ni-eq2} hold for $k \in [K_i,n_i]$. 
Then $(K_i), (n_i)$ satisfy \eqref{lemma-Ki-ni-1}. \\
Let $i \ge 1$  and $k \in [K_i,n_i]$.  By \eqref{lemma-Ki-ni-eq2}, $\left|\dkxy-Ak\right| \le \eps_ik \le Ak/2^i$. 
Furthermore, 
\bua
\dkxy+\dnixyk & \le  & \dkxy+\dnixy - (A - \varepsilon_i)k \quad\text{by \eqref{lemma-Ki-ni-eq1}}\\
& \le  & \dnixy + 2\varepsilon_i k \leq \dnixy +  \frac{2\varepsilon_i\dkxy}{A-\varepsilon_i} \quad\text{by \eqref{lemma-Ki-ni-eq2}}\\
& \le & \dnixy +  \min\{1/2^i,\alpha_i\}\dkxy \quad\text{by \eqref{lemma-Ki-ni-eq0}}.
\eua
As $d(a_k(x)y, a_{n_i}(x)y) = d(a_k(x)y, a_k(x)a_{n_i-k}(T^k x)y) \le \dnixyk$,
\eqref{lemma-Ki-ni-iii-2} holds. 
\end{proof}

Recall that a geodesic space $Y$ is {\it Busemann convex} if given any pair of geodesic 
paths $\gamma_1 : [0, l_1] \to Y$ and $\gamma_2 : [0,l_2] \to Y$ with $\gamma_1(0) = \gamma_2(0)$,
 one has $ d(\gamma_1(tl_1),\gamma_2(tl_2)) \le td(\gamma_1(l_1),\gamma_2(l_2))$ for every $t \in [0,1]$.

\subsection{Proof of Theorem \ref{main-erg-thm}}

Let $x \in E$ be such that $\lim_{n \to \infty}\dnxy/n = A$. For every $i \in \mathbb{N}$ there exist
$s_i > 0$ such that $\alpha_i=\inf_{r \ge s_i} \Psi\left(y,r,1/2^i\right) > 0$ (by (C2)) 
and $p_i \in \mathbb{N}$ such that $\dnxy\ge s_i$ for $n \ge p_i$.  
Apply Lemma \ref{lemma-Ki-ni} to obtain the sequences $(K_i)$ and $(n_i)$ satisfying 
properties (i)-(ii) thereof. For each $j\in\N$, we denote by $\gamma_j$ the geodesic path 
that joins $y$ and $a_{n_j}(x)y$ and for each $j,k\in\N$ we let
\[\mjk=\min\{\dkxy,\dnjxy\}.\]
\noindent {\bf Claim 1}: For all  $j\in \mathbb{N}$ and all $k \in [K_j, n_j]$, 
\beq\label{main-thm-claim1}
d(a_k(x)y, \gamma_j(\mjk) \le \frac{\dkxy}{2^j}. 
\eeq
{\bf Proof of claim:} 
If $\dkxy > \dnjxy$, we apply \eqref{lemma-Ki-ni-iii-2} to get that
\bua
d(a_k(x)y,\gamma_{j}(\dnjxy))&=& d(a_k(x)y,a_{n_j}(x)y)  \leq \frac{\dkxy}{2^j}.
\eua
Assume now that $\dkxy \leq\dnjxy$. Since $k \ge p_j$, we get that $\dkxy\ge s_j$ and so 
$\alpha_j \le \Psi(y,\dkxy, 1/2^j)$.
Then \eqref{lemma-Ki-ni-iii-2} yields that
\bua 
\dkxy + d(a_k(x)y,a_{n_j}(x)y) & \le &    \dnjxy + \alpha_j\dkxy\\
& \le & \dnjxy +  \Psi(y,\dkxy, 1/2^j)\dkxy. 
\eua
Apply now property (C) to obtain that $d(a_k(x)y, \gamma_j(\dkxy)) \le \dkxy/2^j.$ \hfill $\blacksquare$\\[2mm]
For all $i\in\N$, we have that  $n_i\in[K_{i+1},n_{i+1}]$ and $\dnixy \le \dniuxy$, so we can 
apply \eqref{main-thm-claim1} (with $j:=i+1$ and $k:=n_i$) to conclude that 
\beq\label{main-thm-eq-1}
d(\gamma_{i+1}(\dnixy), \gamma_i(\dnixy)) =  d(\gamma_{i+1}(\dnixy), a_{n_i}(x)y) \le \frac{\dnixy}{2^{i+1}}.
\eeq
Fix now $R > 0$ and define $I(R)$ as the smallest integer for which $\dnirxy \ge R$. 
Let $i \ge I(R)$ be arbitrary. Since $(\dnixy)_i$ is nondecreasing, one has that  $\dnixy\ge R$.
By Busemann convexity, we get that 
$d(\gamma_{i+1}(R), \gamma_i(R)) \leq 
(R/\dnixy)d(\gamma_{i+1}(\dnixy), \gamma_i(\dnixy))$
and an application of \eqref{main-thm-eq-1} gives us
\beq\label{main-thm-eq-2}
d(\gamma_{i+1}(R), \gamma_i(R)) \leq  \frac{R}{2^{i+1}}. 
\eeq
Applying repeatedly \eqref{main-thm-eq-2}, we get that for all $m\in\N$, 
\[
d(\gamma_{i+m}(R), \gamma_i(R)) \le \sum_{j=1}^m \frac{1}{2^{i+j}}R \le \frac{R}{2^i}.
\]
Thus, $(\gamma_i(R))_{i \ge I(R)}$ is Cauchy, hence converges to 
$\gamma(R):= \lim_{i \to \infty}\gamma_i(R)$. Furthermore, 
\begin{equation} \label{main-thm-eq-3}
d(\gamma(R), \gamma_i(R)) \le \frac{R}{2^i} \quad \text{for all } i\geq I(R).
\end{equation}
It is easy to see that $\gamma$ is a geodesic ray starting at $y$.  

We shall prove in the sequel that $\lim_{k \to \infty}d(\gamma(Ak), a_k(x)y)/k = 0$. 
Let $k\in\N$. Then there exists $i$ such that $k\in[n_{i-1}, n_i)$, which yields $k\in (K_i, n_i)$. \\[1mm]
{\bf Claim 2}: $|\mik-Ak|\leq Ak/2^i$.\\
{\bf Proof of claim:}  If $\dkxy\leq \dnixy$, then by Lemma \eqref{lemma-Ki-ni}.(ii), we get that 
$|\mik-Ak|=|\dkxy-Ak| \le Ak/2^i$.
If $\dkxy > \dnixy$, since $|\dnixy-An_i| \le An_i/2^{i+1}$ (by Lemma \eqref{lemma-Ki-ni}.(ii))), we have 
that
\[(A-A/2^{i})k\leq (A-A/2^{i+1})n_i\leq \dnixy<\dkxy\leq (A+A/2^i)k.\]
On gets immediately that $|\mik-Ak|=|\dnixy-Ak|\leq Ak/2^i$. \hfill $\blacksquare$\\[2mm]
By \eqref{main-thm-claim1}, $d(a_k(x)y, \gamma_i(\mik)) \le \dkxy/2^i\leq 2Ak/2^i$.
Since  $\dniuxy > Ak$ and $\dnixy \ge \mik$, we have that  $i+1 \ge I(Ak)$ and $i \ge I(\mik)$, so we can apply 
\eqref{main-thm-eq-3} and \eqref{main-thm-eq-2} to get that 
$d(\gamma(Ak), \gamma_{i+1}(Ak))\leq Ak/2^{i+1}$ and
$d(\gamma_{i+1}(\mik), \gamma_{i}(\mik))\leq \mik/2^{i+1}\leq 2Ak/2^{i+1}$.
Thus, 
\bua
d(\gamma(Ak), a_k(x)y) & \le & d(\gamma(Ak), \gamma_{i+1}(Ak)) + d(\gamma_{i+1}(Ak), \gamma_{i+1}(\mik))\\
&&\quad + d(\gamma_{i+1}(\mik), \gamma_{i}(\mik)) + d(\gamma_i(\mik), a_k(x)y)\\
&\leq & \frac{Ak}{2^{i+1}}+ |Ak-\mik|+\frac{2Ak}{2^{i+1}}+\frac{2Ak}{2^i}\leq  \frac{9Ak}{2^{i+1}}.
\eua
This implies that $\ds \lim_{k \to \infty}\frac{1}{k}d(\gamma(Ak), a_k(x)y) = 0$. 
By Busemann convexity it follows easily that the obtained geodesic ray is unique.

\mbox{ } 

\noindent
{\bf Acknowledgements:} \\[1mm] 
Lauren\c tiu Leu\c stean was supported by a grant of the Romanian 
National Authority for Scientific Research, CNCS - UEFISCDI, project 
number PN-II-ID-PCE-2011-3-0383. \\
Adriana Nicolae was supported by a grant of the Romanian
Ministry of Education, CNCS - UEFISCDI, project number PN-II-RU-PD-2012-3-0152.\\
The authors are grateful to Bo\.{z}ena Pi\c{a}tek for providing useful comments on the paper.

\end{document}